\documentclass[final,5p, twocolumn]{elsarticle}
\usepackage{geometry}                % See geometry.pdf to learn the layout options. There are lots.
\usepackage{graphicx}
\usepackage{amssymb, amsmath, amsthm}
\usepackage{epstopdf}
\usepackage{bm}
\usepackage{algorithm}
\usepackage{threeparttable,booktabs}
\usepackage{multicol}
\usepackage{appendix}

\newtheorem{theorem}{Theorem}[section]

\DeclareMathOperator*{\argmax}{arg\,max}

\begin{document}
\begin{frontmatter}

% "Title of the paper"
\title{Perfect Sampling and Gradient Simulation for Fork-Join Networks}
%\runtitle{Perfect Sampling and Gradient Simulation for a Fork-Join Network}

% indicate corresponding author with \corref{}

\author[focal,rvt]{Xinyun Chen\corref{cor1}}
    \ead{xinyun.chen@whu.edu.cn}
\author[rvt]{Xianjun Shi}
    \ead{xianjun.shi@stonybrook.edu}

\cortext[cor1]{Corresponding author}
%\fntext[fn1]{This is the specimen author footnote.}
%\fntext[fn2]{Another author footnote, but a little more longer.}

\address[focal]{Economics and Management School, Wuhan University, Wuhan, Hubei, 430072, China}
\address[rvt]{Department of Applied Mathematics and Statistics,
Stony Brook University, Stony Brook, NY 11794, USA} 

\begin{abstract}
Fork-join network is a class of queueing networks with applications in manufactory, healthcare and computation systems. In this paper, we develop a simulation algorithm that (1) generates i.i.d. samples of the job sojourn time, jointly with the number of waiting tasks, exactly following the steady-state distribution, and (2) unbiased estimators of the derivatives of the job sojourn time with respect to the service rates of the servers in the network. The algorithm is designed based on the Coupling from the Past (CFTP) and Infinitesimal Perturbation Analysis (IPA) techniques. Two numerical examples are reported, including the special 2-station case where analytic results on the steady-state distribution is known and a 10-station network with a bottleneck.  

\end{abstract}


%\begin{keyword}
%\kwd{}
%\kwd{}
%\end{keyword}
\end{frontmatter}

\section{Introduction}
A fork-join network is a special type of queueing networks. When a \textit{job} arrives at the network, it splits into several parts, which we call the \textit{tasks}, to be served in different service stations. After being served, the tasks join together again to form the output. This step is usually called synchronization of the tasks. There are two types of synchronization: exchangeable synchronization(ES) and non-exchangeable synchronization (NES). In our paper, we shall focus on fork-join networks with NES, which we shall explain in one moment. For the definition and examples of ES fork-join network, please see \cite{Pang16} and the references therein.

In an NES fork-join network, each task is tagged with the job it comes from and  can get synchronized only when all the other tasks from the same job have been served. Such type of networks have many application in healthcare and parallel computing system (see \cite{Pang16} and the references therein). For example,  in the Map-Reduce scheduling \cite{DG04}, a large data set (job) is splitted into several small sets (tasks). The tasks are processed in different servers, and then the computing results of the tasks are sent to a single server to form the final output of the job. As each job contains different data, tasks from different jobs can not be mixed up. Another example is the procedure of diagnosis \cite{AIMMT15}. A doctor may ask one patient to do several medical examinations. The results of all the examinations must be ready before the doctor can make the diagnosis. Besides, the results from different patients can not be mixed up.

In this paper, we consider a fundamental NES model where the number of service station and the number of tasks in each job are both equal to a constant $K$. Two important performance measures of such networks are the job sojourn time, that is from the arrival time of the job to the time when all its tasks got synchronized, and the number of tasks waiting in the system. Both of them are random variables and their steady-state distribution depict the long-run performance of the system. Unfortunately,  even for this fundamental model, there is no analytic result on the steady-state distribution except for the very special case when the system is Markovian and has only 2 stations \cite{FH84, NT93}. Moreover, in engineering problems such as optimal allocation of the service resource, one also needs to know the sensitivities of the performance measures with respect to the model parameters.

We shall deal with these performance measure and sensitivity analysis problems using the perfect-sampling technique. In particular, under mild condition, we develop an algorithm that can generate i.i.d. sample following exactly the steady-state (joint) distribution of the job sojourn time and the number of waiting tasks in the fork-join network. In addition, the algorithm generates unbiased estimators of the derivatives of the mean job sojourn time with respect to the service rates of the stations.

Our work is closely related to the recently growing literature on the perfect-sampling algorithms for queues and queueing networks \cite{BC14, BDGP15, CK14, MT06, S12}. In \cite{BC14}, the authors develop a perfect-sampling algorithm for the virtual waiting time of stochastic fluid networks. We use their algorithm to simulate the job sojourn time and extend it to the number of waiting tasks. The gradient simulation part is based on the IPA approach developed in \cite{G92}.

The rest of paper is organized as follows. In Section 2, we give the mathematic description of the model. Section 3 gives detailed explanation and the proofs of the algorithm. Numerical results are reported in Section 4.
\section{Notation, Model and Problem Setting}
Through out the paper, we shall use boldface to represent a vector in $\mathbb{R}^K$ for some $K\geq 2$. For example, $\mathbf{J}=(J_1, J_2,...,J_K)$ is a vector in $\mathbb{R}^K$ and $J_k$ is the $k$-th component of $\mathbf{J}$. $\mathbf{1}=(1,1,...,1)\in \mathbb{R}^K$ and $\mathbf{e}^k$ represents the vector in $\mathbb{R}^K$ whose $k$-th component equals to 1 and all other components are 0.

%(To boldface $\mu$ or other greeks, use the code \text{$\backslash$boldsymbol}.)

We consider a fork-join network with $K$ parallel service stations indexed as station 1 to station $K$. Jobs arrive according to some renewal process with i.i.d. inter-arrival times $\{I(n)\}$, i.e., $I(n)$ is the time between the arrival of the $n-1$-th and $n$-th job. Upon arrival,  each job split into $K$ tasks and the $k$-th task is sent to and will be processed by service station $k$.  We shall denote the $k$-th task from the $n$-th job as $T_k(n)$. Let $J_{k}(n)$ be the service requirement (or time) of $T_{k}(n)$. For fixed $n$, $J_{k}(n)$ can be correlated and follow different distributions, but the sequence of random vectors $\{\mathbf{J}(n)\}$ are i.i.d. in $\mathbb{R}^K$. Besides, $\{I(n)\}$ and $\{\mathbf{J}(n)\}$ are independent. After they finish service, tasks will join a ``unsynchronized queue" until all the other tasks from the same job finish service. Once upon the last task finishes service, the $K$ tasks from the same job will get synchronized and the job leaves the system immediately. 

It is known in \cite{PJ89} that when $\max_k E[J_{k}(n)]<E[I(n)]$, the fork-join network is stable and its workload process has a unique stationary distribution. So are the sojourn time of each job and the total number of waiting tasks in the system. Our goal is to develop algorithms to generate i.i.d. samples following exactly the stationary (joint) distribution of the job sojourn time and the number of waiting tasks in the system, and to generate unbiased estimators for the derivatives of the mean job sojourn time with respect to the service rates of the stations. 

In the rest of the paper we shall impose the following assumptions:\\

\noindent\textbf{Assumptions}\\
%\item $I(n)$ has infinite support.
(A1) $\max_k E[J_{k}(n)]<E[I(n)]$ so that the system is stable.\\
(A2) There exists $\theta_k>0$ for $k=1,2,...,K$ such that
$E[\exp(\sum_k \theta_kJ_{k}(n))]<\infty$.
%\item For all $k$, $J_{k}(n)$ has continuous density.

\section{Algorithm Description}
\subsection{Perfect sampling for the job sojourn time}
Recall that $T_{k}(n)$ is the $k$-th task from the $n$-th job. Let $\{W_{k}(n)\}$ be the waiting time of task $T_{k}(n)$ before entering service at station $k$. Then we have
\begin{equation}\label{eq:W}
 W_{k}(n)=(W_{k}(n-1)+J_{k}(n-1)-I(n))^+,\end{equation}
and the total sojourn time of task $T_{k}(n)$ at station $k$ is 
$$S_{k}(n)=W_{k}(n)+J_{k}(n).$$
In other words, the $k$-th task will enter the unsynchronized queue $S_{k}(n)$ units of time after the job enter the system. Since there is no synchronization time in our setting, the job will leave the system immediately once all its $K$ tasks enter the unsynchronized queue. Therefore, the total sojourn time of the $n$-th job is simply
$$S(n)=\max_{1\leq k\leq K}S_{k}(n).$$

Let $\mathbf{W}(n)=(W_1(n),...,W_K(n))$. Then, following \eqref{eq:W}, $\{\mathbf{W}(n)\}$ is a Markov process in $\mathbb{R}^K$. In \cite{BC14}, the authors construct a stationary version of $\{\mathbf{W}^*(-n)\}_{n\geq 0}$ \textit{backward in time} that is expressed as the difference of a random walk and its maximum in the future. In detail, for all $n\geq 0$
$$\mathbf{W}^*(-n)=\mathbf{M}(n)-\mathbf{R}(n),$$
where $\mathbf{R}(n)$ is a random walk in $\mathbb{R}^K$ such that
$$\mathbf{R}(n)=\mathbf{R}(n-1)+\mathbf{J}(n)-I(n)\mathbf{1}, \mathbf{R}(0)=0.$$
And,
$$\mathbf{M}(n)=\max_{m\geq n}\mathbf{R}(n),$$
where the maximum is taken component by component.\\

Using Algorithm 4 in \cite{BC14}, we can simulate i.i.d. sample paths of $\{\mathbf{W}^*(n)\}_{n=-N}^0$ for any $N<\infty$. In particular, $\mathbf{W}^*(0)$ follows the stationary distribution of the task waiting time. Let $\mathbf{J}(0)$ follows the distribution of $\mathbf{J}$ and is independent of $\{\mathbf{R}(n)\}_{n\geq 1}$, representing the job requirement of the 0-th job. Then, 
$$S^*(0)=\max_{1\leq k\leq K}(W^*_{k}(0)+J_{k}(0))$$ 
is the sojourn time of the 0-th job in the stationary sample path, and therefore follows the stationary distribution of the job sojourn time in the fork-join system. \\

\subsection{Perfect sampling for the number of tasks}
From the discrete sequence of $\{\mathbf{W}^*(n)\}_{n\leq 0}$, we can derive a stationary sample path the fork-join system in continuous time. In the first step, we replace $I(1)$ with $I^*(1)$ following the equilibrium distribution of $I(n)$. Let $A(-n)=-\sum_{m=1}^n I(i)$. Then the sample path of the fork join system in  continuous time is as follows. The $n$-th job \textit{in the past}  arrive at time $A(-n)$ and brings tasks with service requirement $\mathbf{J}(n)$ to the system and its sojourn time is $S^*(-n)$ and each of its task at the k-th service station waits $W^*_{k}(-n)$ units of time before entering service.

Now we show how to compute the number of tasks in the system using the sample path information. At time $t$, let $Q_k(t)$ be the number of tasks in the $k$-th service station, including the one in service, and $D_k(t)$ be the number of tasks in the $k$-th unsynchronized queue. Then,  $Q_k(t)$ and $D_k(t)$ can be expressed as:
\begin{footnotesize}\begin{eqnarray*}
\begin{cases}
Q_k(t)=&\sum_{n\geq 1}1(A(-n)+W^*_{k}(-n)+J_{k}(n)>t\\&\text{ and }A(-n)<t),\\ 
D_k(t)=&\sum_{n\geq 1}1(A(-n)+W^*_{k}(-n)+J_{k}(n)<t\\
& \text{ and }A(-n)+S^*(-n)>t).
\end{cases}
\end{eqnarray*}\end{footnotesize}
As the service is under a FCFS discipline, $A(-n)+W^*_{k}(-n)+J_{k}(n)$ is strictly decreasing in $n$ and goes to $-\infty$ as $n\to\infty$. So we will eventually see $A(-N)+W^*_{k}(-N)+J_{k}(N)<0$ for all $k$ at some finite $N$ and then we have
\begin{footnotesize}\begin{equation}
\begin{cases}\label{eq: QD}Q_k(0)=&\sum_{n=1}^{N}1(A(-n)+W^*_{k}(-n)+J_{k}(n)>0)\\
D_k(0)=&\sum_{n=1}^{N}1(A(-n)+W^*_{k}(-n)+J_{k}(n)<0\\&\text{ and }
A(-n)+S^*(-n)>0).\end{cases}\end{equation}\end{footnotesize}
In other words, we just need to simulate $\{\mathbf{W}^*(n)\}$ backwards in time until $n=N$ such that $A(-N)+W^*_{k}(-N)+J_{k}(N)<0$ for all $k$ (and hence $A(-N)+S^*(-N)<0$) and stop. Then, we compute $Q_k(0)$ and $D_k(0)$ according to \eqref{eq: QD}, which jointly follow the stationary distribution of the number of tasks in service stations and in the unsynchronized queues. 

\subsection{Gradient simulation}
In some engineering problem, we want to know the impact of the service resource allocation among different stations on the mean sojourn time of jobs. This requires a sensitivity analysis of the job sojourn time with respect to the the service capacity of each station. In order to represent the service capacity mathematically, we introduce $\mu_k$ as the ``service rate" of service station $k$ for all $k=1,2,...,K$, such that the service time of a task  with service requirement of $J_{k}(n)$ at station $k$ is $J_{k}(n)/\mu_k$. Under this new setting, the stability condition (A1) now becomes: 
$$
\begin{array}{lc}
\text{(A3)  } &\max _{1\leq k\leq K} E[J_{k}(n)/\mu_k] < E[I(n)],
\end{array} $$which we shall impose throughout this section.

The sensitivity analysis involves computing the derivatives $\frac{\partial E[S^*(0)]}{\partial \mu_k}$ for all $k=1,2,...,K$. In this section, we shall develop a simulation algorithm generating i.i.d. samples from some random vector $\mathbf{H}^*=(H^*_1,...H^*_K)$, such that $E[H^*_k]=\frac{\partial E[S^*(0)]}{\partial \mu_k}$ for all $1\leq k\leq K$. 

Suppose that $\mathbf{W}^*(\boldsymbol{\mu})$ follows the stationary distribution of $\mathbf{W}(n)$ under the service rate vector  $\boldsymbol{\mu}=(\mu_1,...,\mu_K)$. Suppose $\mathbf{J}(0)$ follows the job requirement distribution and is independent of $\mathbf{W}^*(\boldsymbol{\mu})$, then
 \begin{footnotesize}\begin{equation*}
 S^*(\boldsymbol{\mu}; 0) = \max_{1\leq k\leq K}(W^*_{k}(\boldsymbol{\mu})+J_{k}(0)/\mu_k)\doteq f(\mathbf{W}^*(\boldsymbol{\mu}), \mathbf{J}(0), \boldsymbol{\mu}),
 \end{equation*}\end{footnotesize}
follows the stationary distribution of the job sojourn time under the service rate vector $\boldsymbol{\mu}$.

Suppose now we have a stationary version of $\{\mathbf{W}^*(n)\}_{n\leq 0}$ as generated in Section 3.1. For each $k$, define $\tau_k=\max\{n: W^*_{k}(n)=0\}>-\infty$ w.p.1 given Assumption (A3). The following result gives our construction of the gradient estimator. To simplify the expression, we define     $$h(\mathbf{W}, \mathbf{J}, \boldsymbol{\mu}) \doteq \argmax_{1\leq l\leq K} W_{l}+J_{l}/\mu_l.$$

\begin{theorem}\label{thm}
For all $1\leq k\leq K$:
\begin{enumerate}
\item $$V_k^*=-\sum_{n=1}^{-\tau_k}\frac{J_{k}(n)}{\mu_k^2}$$
is an unbiased estimator of $\frac{\partial W^*_k(0)}{\partial \mu_k}$ such that $E[V^*_k]=\frac{\partial W^*_k(0)}{\partial \mu_k}<\infty$.
\item
\begin{footnotesize}\begin{equation}\label{eq: H}H^*_k\doteq1(k=h(\mathbf{W}^*(0), \mathbf{J}(0), \boldsymbol{\mu})  )\left(V^*_k-\frac{J_{k}(0)}{\mu_k^2}\right)\end{equation}\end{footnotesize}
is an unbiased estimator for $\frac{\partial E[S^*(\boldsymbol{\mu}; 0)]}{\partial \mu_k}$.
\end{enumerate}
\end{theorem}
\textbf{Remark:} Under the assumption that $J_{k}(n)$ are continuous random variables, \begin{footnotesize}$P(|\argmax_{1\leq k\leq K} (W_{k}^*(0)+J_{k}(0)/\mu_k) |~ >1)=0$\end{footnotesize}, so \begin{footnotesize}$1(k=h(\mathbf{W}^*(0), \mathbf{J}(0), \boldsymbol{\mu}))$\end{footnotesize} is well defined w.p.1.)
\begin{proof}[Proof of Theorem \ref{thm}]
Since the tasks are served independently with no feedback, for each $k$, $\{W_{k}(n)\}$ is itself a Markov process in $\mathbb{R}$ satisfying the recursion:
\begin{footnotesize}\begin{eqnarray*}
W^*_{k}(n) &=& (W^*_{k}(n-1)+J_{k}(n-1)/ \mu_k-I(n))^+\\
&\doteq&\phi(W^*_{k}(n-1),\mu_k).
\end{eqnarray*}\end{footnotesize}
Applying the IPA results in \cite{G92},
\begin{footnotesize}$$\sum_{j=0}^{\infty}\left(\prod_{i = n-j}^{n-1} \frac{\partial\phi}{\partial W}(W^*_{k}(i),\mu_k)\right)\frac{\partial \phi}{\partial \mu}(W^*_{k}(n-j-1),\mu_k)$$\end{footnotesize}
is an unbiased estimator if this infinity sum is well-defined. We can compute that 
\begin{align*} 
\frac{\partial\phi}{\partial W}(W^*_{k}(n),\mu_k)&=1(W^*_{k}(n+1)>0),\\
\frac{\partial \phi}{\partial \mu}(W^*_{k}(n),\mu_k)&=-1(W^*_{k}(n+1)>0)\frac{J_{k}(-n)}{\mu_k^2}.\end{align*}
Under Assumptions (A3),   $\tau_k=\max\{n: W^*_{k}(n)=0\}<0$. Therefore, the infinity sum is well-defined and equal to
 $$-\sum_{n=\tau_k}^{0}\frac{J_{k}(n)}{\mu_k^2}\doteq V_k^*,$$
which is an unbiased estimator for $\frac{\partial W^*_k(0)}{\partial \mu_k}$. 

Besides, the IPA construction also indicates a family of $\{\mathbf{W}^*(\boldsymbol{\eta}; 0):\boldsymbol{\eta}\in (\boldsymbol{\mu}-\delta\mathbf{1},\boldsymbol{\mu}+\delta\mathbf{1})\}_{n\leq 0}$, for some $\delta>0$ small enough, that are coupled such that for each $1\leq k\leq K$,
$$\frac{W^*_{k}(\mu_k+\theta; 0)-W^*_{k}(\mu_k; 0)}{\theta}\to V^*_k, \text{w.p.1 as }\theta\to 0$$
in the $\sigma$-field generated by $\{I(n), \mathbf{J}^*(n)\}_{n\leq 0}$. Now we consider such a family of $\{\mathbf{W}^*(\boldsymbol{\eta}; 0):\boldsymbol{\eta}\in (\boldsymbol{\mu}-\delta\mathbf{1},\boldsymbol{\mu}+\delta\mathbf{1})\}$. 
By definition,
\begin{scriptsize}\begin{align*}
&\frac{\partial E[S^*(\boldsymbol{\mu};0)]}{\partial \mu_k}\\
=&\lim_{\theta\to 0}\frac{E[f(\mathbf{W}^*(\boldsymbol{\mu}+\theta\cdot \mathbf{e}^k; 0), \mathbf{J}(0),\boldsymbol{\mu}+\theta\cdot \mathbf{e}^k)]-E[f(\mathbf{W}^*(\boldsymbol{\mu}; 0), \mathbf{J}(0), \boldsymbol{\mu})]}{\theta}.\end{align*}\end{scriptsize}
As $E[J_k],$ $E[\mathbf{W}^*(\boldsymbol{\mu}; 0)]$, $E[\mathbf{W}^*(\boldsymbol{\mu}+\theta\cdot \mathbf{e}^k; 0)]<\infty$, we have
\begin{scriptsize}\begin{align}\label{error}
&\frac{E[f(\mathbf{W}^*(\boldsymbol{\mu}+\theta\cdot \mathbf{e}^k; 0), \mathbf{J}(0),\boldsymbol{\mu}+\theta\cdot \mathbf{e}^k)]-E[f(\mathbf{W}^*(\boldsymbol{\mu}; 0), \mathbf{J}(0), \boldsymbol{\mu})]}{\theta}\nonumber\\
=&E\left[1(k=h(\mathbf{W}^*(0), \mathbf{J}(0), \boldsymbol{\mu}))\left\{\left[\frac{W_{k}^*(\mu_k+\theta; 0)-W_{k}^*(\mu_k; 0)}{\theta}\right.\right.\right.\nonumber\\
&\left.\left.\left.+\frac{\frac{J_k(0)}{\mu_k+\theta}-\frac{J_k(0)}{\mu_k}}{\theta}\right]+\sum_{j\neq k}1(j=h(\mathbf{W}^*(0), \mathbf{J}(0), \boldsymbol{\mu}+\theta\cdot \boldsymbol{e}^k))\right.\right.\nonumber\\
&\left.\left.\cdot\left[\frac{W^*_{j}(\mu_j; 0)-W_{k}^*(\mu_k+\theta; 0)}{\theta}+\frac{\frac{J_j(0)}{\mu_j}-\frac{J_k(0)}{\mu_k+\theta}}{\theta}\right]\right\}\right]
\end{align}\end{scriptsize}
Since $\theta\to 0$, we get
\begin{scriptsize}$$P(k=h(\mathbf{W}^*(0), \mathbf{J}(0), \boldsymbol{\mu})\text{ and }k\neq h(\mathbf{W}^*(0), \mathbf{J}(0), \boldsymbol{\mu}+\theta\cdot \boldsymbol{e}^k))\to 0.$$\end{scriptsize}
As \begin{footnotesize}$E[J_k(0)], E[\mathbf{W}^*(\boldsymbol{\mu}; 0)], E[\mathbf{W}^*(\boldsymbol{\mu}+\theta\cdot \mathbf{e}^k; 0)]<\infty$,\end{footnotesize} we conclude that
the expectation of the terms in \eqref{error} goes to 0 as $\theta\to 0$ by the dominance convergence theorem. 
On the other hand, we have
\begin{scriptsize}$$\frac{W_{k}^*(\mu_k+\theta; 0)-W_{k}^*(\mu_k; 0)}{\theta}\to V^*_k \text{ and }\frac{\frac{J_k(0)}{\mu_k+\theta}-\frac{J_k(0)}{\mu_k}}{\theta}\to -\frac{J_k(0)}{\mu_k^2}.$$\end{scriptsize}
By the dominance convergence theorem, we conclude that
\begin{scriptsize}\begin{align*}
&\frac{E[f(\mathbf{W}^*(\boldsymbol{\mu}+\theta\cdot \mathbf{e}^k; 0), \mathbf{J}(0),\boldsymbol{\mu}+\theta\cdot \mathbf{e}^k)]-E[f(\mathbf{W}^*(\boldsymbol{\mu}; 0), \mathbf{J}(0), \boldsymbol{\mu})]}{\theta} \\
\rightarrow~&E\left[1(k=h(\mathbf{W}^*(0), \mathbf{J}(0), \boldsymbol{\mu}))\left(V^*_k-\frac{J_{k}(0)}{\mu_k^2}\right)\right].
\end{align*}\end{scriptsize}
%$ E\left[1(k=\argmax_{1\leq l\leq K} W^*_{l}(0)+J_{l}(0)/\mu_l )\left(V^*_k-\frac{J_{k}(0)}{\mu_k^2}\right)\right].$

\end{proof}

We close this section by summarizing the algorithm:

\textbf{Algorithm: Perfect sampling and gradient simulation for FJQ}

1. Simulate $\{\mathbf{W}(-n)^*\}_{n\geq 0}$ backwards in time jointly with i.i.d. sequences $\{\mathbf{J}(n)\}$ and $\{I(n)\}$ until $n=N$ such that for all $k=1,2,..,K$, $A(-N)+W^*_{k}(-N)+J_{k}(N)=0$.

2. For $0\leq n\leq N$, compute $S^*(-n)=\max_{1\leq k\leq K} (W^*_{k}(-n)+J_{k}(n))$.

3. Compute $Q_k(0)$ and $D_k(0)$ according to \eqref{eq: QD}

4.  Let $k_0=\argmax_{1\leq l\leq K} (W^*_{l}(0)+J_{l}(0)/\mu_l)$. Get the value of $\tau_{k_0}=\max\{n: W^*_{k_0}(n)=0\}$. (One can check that $\tau_k\geq -N$ for all $k$.) Compute a vector $(\mathbf{H}^*)_{1\leq k\leq K}$ according to \eqref{eq: H}.

5. Output $S^*(0)$, $(Q_k(0), D_k(0))_{k=1}^K$ and $\mathbf{H}^*$.

\section{Numerical Experiments}
We implement the algorithm in MATLAB performed on a PC with an Intel Core i7-4790 CPU 3.60GHz, 16.00 GB of RAM. We first test the correctness of the algorithm by simulating a simple 2-station example where some analytic results on the stationary distributions are available. Then we simulate a 10-station example illustrating the efficiency of the algorithm.

\subsection{The 2-station case}
We consider a 2-station example where jobs arrive according to a Poisson process of rate $\lambda=1$. Job requirements $J_{1}(n)$ and $J_{2}(n)$ are independent and following an exponential distribution of rate $\mu$. In this simple case, the stationary expectation of the job sojourn time $S^*(0)$ and the number of unsynchronized tasks in the system $D^*(0)$ have closed form expressions as given in \cite{NT93} and \cite{ FH84}:
\begin{scriptsize}\begin{equation}\label{es}
E[S^*(0)] = \frac{12\mu - \lambda}{8\mu(\mu-\lambda)}\text{ and } E[D^*(0)] = \frac{\lambda(4\mu-\lambda)}{4\mu(\mu-\lambda)}.
\end{equation}\end{scriptsize}

\begin{table*}[t]
\centering
\caption{Simulation Results for $E[S^*(0)]$ and $E[\mathbf{D}^*(0)]$: $K = 2$,  $\{J_{1}(n), J_{2}(n) \}$ are i.i.d. exponential r.v.'s of rate $\mu$}
\label{213}
\begin{tabular}{|l|l|l|l|l|}
\hline 
%NO. & 1 & 2 & 3 & 4 \\ \hline
 
 $\mu$& 1.8000 & 1.4000& 1.1000 & 1.0600\\ \hline
 
 $1-\rho$& 0.4444 & 0.2857  & 0.0909 & 0.0566\\ \hline
 
True E[$S^*(0)$]& 1.7882 & 3.5268 & 13.8636& 23.0346 \\ \hline
 
 Simulated E[$S^*(0)$]& 1.7901$\pm$0.0279 & 3.5094$\pm$0.0555& 14.0178$\pm$0.2218& 23.2005$\pm$0.3674\\ \hline
 
 True E[$D^*(0)]$& 1.0764 & 2.0536& 7.7273& 12.7358\\ \hline
 
 Simulated E[$D^*(0)$]& 1.0756$\pm$0.0275 & 2.0396$\pm$0.0483 & 7.8104$\pm$0.1697& 12.9087$\pm$0.2796\\ \hline
 
 Running time(s)  & 3.3032 & 5.1520& 37.5992& 134.9230\\ \hline 

\end{tabular}
\end{table*}

\begin{table*}[t]
\centering
\caption{Simulation Results for $\partial E[S^*(0)]/\partial \mu$: $K = 2$,  $\{J_{1}(n), J_{2}(n) \}$ are i.i.d. exponential r.v.'s of rate $\mu$}
\label{213}
\begin{tabular}{|l|l|l|l|l|}
\hline 
%NO. & 1 & 2 & 3 & 4 \\ \hline
 
 $\mu$& 1.8000 & 1.4000& 1.1000 & 1.0600\\ \hline
 
 $1-\rho$& 0.4444 & 0.2857  & 0.0909 & 0.0566\\ \hline
 
True $\partial E[S^*(0)]/\partial \mu$& -2.1870 & -8.6575 & -137.6033 & -382.0557\\ \hline
 
 Simulated $\partial E[S^*(0)]/\partial \mu$& -2.1964$\pm$0.0527 & -8.5430$\pm$0.2147 & -136.4920$\pm$3.8702& -376.8422$\pm$10.6424\\ \hline
 
 Running time(s)  & 4.6017 & 6.9625& 50.4201& 144.9027\\ \hline 

\end{tabular}
\end{table*}

Table 1 compares the simulation estimations, the true values of  $E[S^*(0)]$ and $E[D^*(0)]$ for variety of $\mu$. For each $\mu$, we do 10000 round of simulation algorithm and the total running time is also reported in Table 1. To test the performance of our algorithm in heavy traffic, we let $1-\rho = 1- \lambda/\mu$ approach 0.

To test the validity of simulated $95\%$ CI, we simulated 1000 independent $95\%$ CIs for $S^*(0)$ and $D^*(0)$ when $\mu= 1.4$. For each CI, we generated 10000 independent sample paths. Out of the 1000 CIs, 956 cover the true value 3.5268($S^*(0)$) and 2.0536($D^*(0)$), which indicates that our estimator is unbiased.

Then we implement the gradient simulation algorithm to estimate $\frac{\partial E[S^*(0)]}{\partial\mu_1}$ and $\frac{\partial E[S^*(0)]}{\partial \mu_2}$. In the special case where the distribution of the service times at the two stations follows two independent exponentials with the same rate $\mu$ (i.e. $J_i\sim \exp(1)$ for $i = 1, 2$ and the service rate is $\mu$), using \eqref{es} we get an explicit expression for $\frac{\partial E[S^*(0)]}{\partial\mu}$:
$$\frac{\partial E[S^*(0)]}{\partial\mu} = \frac{2\lambda\mu - \lambda^2-12\mu^2}{8\mu^2(\mu-\lambda)^2},$$

Since $\mu_1 = \mu_2 = \mu$, we have
$$\frac{\partial E[S^*(0)]}{\partial\mu} = \frac{\partial E[S^*(0)]}{\partial \mu_1}+\frac{\partial E[S^*(0)]}{\partial \mu_2}.$$
Recall that  $H^*_1$ and $H^*_2$ are the unbiased estimators for $\frac{\partial E[S^*(0)]}{\partial\mu_1}$ and $\frac{\partial E[S^*(0)]}{\partial\mu_2}$ generated by our algorithm. Then, $H^*_1+H^*_2$ should be an unbiased estimator for $\frac{\partial E[S^*(0)]}{\partial\mu}$. Table 2 compares the simulation estimations and the true values of  $\frac{\partial E[S^*(0)]}{\partial\mu}$ for a variety values of $\mu$. For each $\mu$, we do 10000 round of simulation algorithm and the total running time is also reported in Table 2.

We test the validity of simulated $95\%$ CI by simulating 1000 independent $95\%$ CIs for $\partial E[S^*(0)]/\partial\mu$ when $\mu = 1.8$, whose theoretical value is -2.1870. Each $95\%$ CI is generated by sampling 10000 independent sample paths. Out of the 1000 CIs, 954 of them cover the true value, which shows the validity of our CIs.

\subsubsection{K=10}
We implement the case where jobs arrive according to a Poisson process of rate 1 and $\{J_{k}(n)\}$ are independently exponentially distributed of rate  $\mu_k=2-0.05k$ for $1\leq k \leq 10$. In this example, Station 10 is a "bottleneck" as its service rate is the smallest and the 10-th task is likely to be last to finish among all the tasks. Using our algorithm, we first estimate $E[S^*(0)]$ and $E[D_k^*(0)]$ for $1\leq k\leq 10$ based on 10000 independent round of simulation. The simulation estimation and the CI's are reported in Table 3. From the simulation results, we can see that the unsynchronized queue corresponding to station 10 ($E[D_k^*(0)]$) is the shortest.
\begin{table}[t]
\centering
\caption{Simulation Results for $ E[S^*(0)]$ and $E[\mathbf{D}^*(0)]$ ($K = 10$ and $\lambda = 1$)}
\label{10g}
\scalebox{0.9}{
\begin{tabular}{|c|cc|}
\hline

\multicolumn{3}{| l |}{$ E[S^*(0)]\pm 95\%$ CI: 3.8452$\pm$0.0384} \\ \hline\hline

$k$ & $\mu_k$ & Simulated $E[D_k^*(0)]$ \\ \hline
 
 1& 2.0000 &2.5445$\pm$0.0470 \\
 
2& 1.9500  &  2.4872$\pm$0.0469 \\
 
 3& 1.9000&  2.4266$\pm$0.0469\\
 
 4& 1.8500& 2.3554$\pm$0.0457 \\
 
 5& 1.8000& 2.2885$\pm$0.0448 \\
 
 6& 1.7500 & 2.1854$\pm$0.0442 \\
 
 7& 1.7000& 2.1194$\pm$0.0437 \\
 
 8& 1.6500& 2.0072$\pm$0.0420 \\
 
 9& 1.6000&1.8433$\pm$0.0395 \\
 
 10& 1.5500& 1.7216$\pm$0.0383\\ \hline
 
 \multicolumn{3}{| l |}{Running time(s): 8.7201} \\ \hline
  
\end{tabular}}
\end{table}

Table 4 reports the gradient simulation results. From the estimated derivative, we can also see that station 10 is the bottleneck station in the sense that $E[S^*(0)]$ is mostly sensitive to service rate $\mu_{10}$.\\

\begin{table}[t]
\centering
\caption{Simulation Results for $\partial E[S^*(0)]/\partial\boldsymbol{\mu}$ ($K = 10$ and $\lambda = 1$)}
\label{10g}
\scalebox{0.9}{
\begin{tabular}{|c|cc|}
\hline

$k$ & $\mu_k$ & Simulated $\partial E[S^*(0)]/\partial \mu_k$\\ \hline
 
 1& 2.0000 &-0.1159$\pm$0.0138 \\
 
2& 1.9500  &  -0.1488$\pm$0.0164\\
 
 3& 1.9000&  -0.1788$\pm$0.0187\\
 
 4& 1.8500& -0.2520$\pm$0.0252 \\
 
 5& 1.8000& -0.2955$\pm$0.0260 \\
 
 6& 1.7500 & -0.4248$\pm$0.0354 \\
 
 7& 1.7000&  -0.5390$\pm$0.0417\\
 
 8& 1.6500& -0.7325$\pm$0.0528\\
 
 9& 1.6000& -0.9924$\pm$0.0651\\
 
 10& 1.5500& -1.3449$\pm$0.0793\\ \hline
  
  \multicolumn{3}{| l |}{Running time(s): 7.7163} \\ \hline
  
\end{tabular}}
\end{table}

\noindent \textbf{Acknowledgement}

This paper is based upon work supported by the National
Science Foundation under Grant No.CMMI 1538102.

\end{document}